\documentclass[12pt]{amsart}

\usepackage{amssymb}

\def\Hxx{\frac{\partial^2 }{\partial x^2}}

\def\Haa{\frac{\partial^2 }{\partial a^2}}

\def\gtupn{\mathbb S_n(\mathbb R^g)}
\def\gxtupn{\gtupn}
\def\gatupn{\gtupn}
\def\ggtupn{\mathbb S_n(\mathbb R^g\times\mathbb R^g)}
\def\gtup{\mathbb S(\mathbb R^g)}
\def\ggtup{\mathbb S(\mathbb R^g\times \mathbb R^g)}

\def\Rax{\mathbb R\langle a,x \rangle}

\def\degx{\text{deg}_x \ }
\def\dega{\text{deg}_a \ }
\def\and{ \ \text{and} \ }

\newtheorem{theorem}{Theorem}[section]
\newtheorem{lemma}[theorem]{Lemma}
\newtheorem{corollary}[theorem]{Corollary}
\newtheorem{proposition}[theorem]{Proposition}

\newtheorem{remark}[theorem]{Remark}
\newtheorem{example}[theorem]{Example}
\newtheorem{conjecture}[theorem]{Conjecture}

\begin{document}

\title[Partial Convexity]{Non-Commutative Partial Matrix Convexity}


\author[Hay]{Damon M. Hay}
\address{Department of Mathematics and Statistics \\
    University of North Florida
}
\email{damon.hay@unf.edu}

\author[Helton]{J. William Helton${}^1$}
\address{Department of Mathematics\\
  University of California \\
  San Diego}
\email{helton@math.ucsd.edu}

\author[Lim]{Adrian Lim${}^2$}
\address{Department of Mathematics \\
  Vanderbilt  University}
\email{adrian.lim@vanderbilt.edu}

\author[McCullough]{Scott McCullough${}^3$}
\address{Department of Mathematics\\
  University of Florida 
}
\email{sam@math.ufl.edu}



\thanks{${}^1$Research
  supported by the NSF, and the Ford Motor Co.}

\thanks{${}^2$Research supported by  NSF and the Ford Motor Co.}

\thanks{${}^3$Research supported by the NSF grant DMS-0140112.}

\begin{abstract}
  Let $p$ be a  polynomial in the non-commuting variables
  $(a,x)=(a_1,\dots,a_{g_a},x_1,\dots,x_{g_x})$.
  If $p$ is
  convex in the variables $x$, then $p$ has degree two in $x$ and
  moreover, $p$ has the form
$$
  p = L + \Lambda ^T \Lambda,
$$
 where $L$ has degree at most one in $x$
  and $\Lambda$ is a (column) vector which is linear in $x,$
  so that $\Lambda^T\Lambda$ is a both sum of squares
  and homogeneous of degree two.
 Of course the
 converse is true also. Further results involving various
 convexity hypotheses on the $x$ and $a$ variables separately
 are presented.
\end{abstract}

\maketitle

\section{Introduction}
 \label{sec:intro}
   Fix a positive integer $g$
   and let  $a=(a_1,\dots,a_{g})$ and $x=(x_1,\dots,x_{g})$
   denote two classes of non-commuting variables which are
   assumed to be symmetric in the sense explained below,
   and let $\Rax$ denote the
   polynomials in these variables.  Thus an element
   of $\Rax$ is an $\mathbb R$-linear combination of words
   built from $a$ and $x$.

   There is a natural involution ${}^T$ on $\Rax$ which reverses the
   order of a word determined by
\begin{equation*}
   (fg)^T = g^T f^T
\end{equation*}
   for $f,g\in\Rax$ and
\begin{equation*}
  x_j^T=x_j \ \ \ a_j^T=a_j,
\end{equation*}
  for $j=1,2,\dots,g$.

   If the polynomial  $L(a,x)$ has degree at
  most one in $x$ and if $\Lambda(a,x)$ is a
  (column) vector which is
   linear in $x$, then the polynomial
$$
  p= L(a,x)+\Lambda(a,x)^T \Lambda(a,x)
$$
   is convex in $x$, since, for each fixed $a$,
$$
  \frac12(p(a,x)+p(a,y)) -   p(a,\frac{x+y}{2})
     = \Lambda(a,\frac{x-y}{2})^T \Lambda(a,\frac{x-y}{2}).
$$
  The converse is a corollary of the main results of this paper.  We
  remark that it was already shown in \cite{HL} that if a symmetric
  polynomial is convex in $x$, then it must have degree two or less in
 $x$.

  In the remainder of this introduction we introduce the
  terminology and background necessary to state our
  main results on the structure of polynomials which
  satisfy various convexity hypotheses.  The exposition
  is restricted to the case where both the $a$ and $x$
  variables are symmetric, but, for the most part, the
  results go through with the obvious modifications to
  the situation where some of the variables
  are symmetric and others are not. A notable
  example where a convex polynomial arises is the Ricatti inequality,
 \begin{equation*}
    0\preceq -b^T x^2 b + a^T x a +c,
 \end{equation*}
  where $x$ is a symmetric unknown and $a,b,c$
  are not necessarily symmetric knowns.

  Of course there is no need to assume that
  the number $g_a$ of $a$ variables is the same as the
  number $g_x$ of $x$ variables, but it does simplify
  the exposition.  The interested reader should
  have no problem in refining various estimates
  which depend upon $g$ in the case where $g_a\ne g_x$.

\subsection{Non-commutative polynomials}
  A non-commutative polynomial, or simply
  polynomial, $q$ in $g$ non-commuting variables $y=(y_1,\dots,y_g)$
  is a $\mathbb R$-linear combination of words in the letters
  $y=(y_1,\dots,y_g)$. Thus
\begin{equation*}
   q=\sum q_w w,
\end{equation*}
  where the sum is finite, the $w$'s are words, and $q_w\in \mathbb
 R$.

  There is a natural involution ${}^T$ on words
  in $y$ which reverses the order of the product. Namely,
\begin{equation}
 \label{eq:word}
   y_{j_1}y_{j_2}\cdots y_{j_n}=w \mapsto w^T = y_{j_n}\cdots
 y_{j_2}y_{j_1}.
\end{equation}
  This involution naturally extends to polynomials by linearity,
$$
  q^T =\sum q_w w^T,
$$
  and a polynomial $p$ is symmetric if $p=p^T$.  As noted before,
  the conventions here mean that $y_j^T=y_j$ and so in this sense
  the variables themselves are symmetric.

  Polynomials are naturally evaluated at $g$-tuples of symmetric
  matrices. Let $\mathbb S_n$ denote the symmetric $n\times n$
  matrices with real entries and let $\gtupn$ denote
  $g$ tuples $Y=(Y_1,\dots,Y_g)$ where each $Y_j\in\mathbb S_n$.
  There is no requirement  that the $Y_j$ commute. Given
  the word $w$ from equation (\ref{eq:word}),
\begin{equation*}
  w(Y)=Y^w = Y_{j_1}Y_{j_2}\cdots Y_{j_n}
\end{equation*}
 as expected, and of course,
\begin{equation*}
  q(Y)=\sum q_w Y^w.
\end{equation*}

  Note that the involution on polynomials is compatible with
  the transpose operation on matrices so that $q(Y)^T=q^T(Y)$.
  In particular, if $p$ is symmetric, then so is $p(Y)$.

\subsection{Convexity}
  A symmetric polynomials $p$ is convex provided
\begin{equation}
 \label{eq:def-convex}
     [tp(Y)+(1-t)p(Z)]-p(tY+(1-t)Z) \succeq 0
\end{equation}
  for all $n$, all pairs $Y,Z\in\gtupn$ and $0<t<1.$  Here $P\succeq
 0$
  means the square matrix $P$ is positive semi-definite in the
  sense that $P=P^T$ and all of its eigenvalues are non-negative.

  Even in one variable ($g=1$) convexity in this non-commutative
  setting is different from ordinary convexity since there
  is no requirement that $Y$ and  $Z$ commute. Indeed, the
  polynomial $p(y)=y^4$ in the single variable $y$ is
  not convex (for a simple example where convexity
  fails for this $p$ see \cite{HM04}).

  The notion of convexity for a polynomial naturally
  extends to partial convexity; i.e., convexity in
  a subset of the variables.

 \subsection{Domains and Convexity}
  It is natural to consider polynomials
   which are assumed convex only
  on   a subset; i.e., where the inequality of
  equation \ref{eq:def-convex} is to hold
  only for some choices of pairs $Y$ and $Z$.
    As a preliminary,
  it is necessary to discuss {\it non-commutative domains}
  or {\it domains} for short.

  Let $\gtup$ denote the sequence
  $(\gtupn)_n$.  A non-commutative domain $\mathcal D$
  in $\gtup$ is a sequence $\mathcal D=(\mathcal D_n)_n$
  where each $\mathcal D_n \subset \gtupn$ which is
  {\it closed under direct sums} in the sense that if
  $D_j\in  S_{n_j}(\mathbb R^g)$, then
  $D_1\oplus D_2 \in  S_{n_1+n_2}(\mathbb R^g)$.

\subsubsection{Non-commutative domains}
 \label{subsubsec:NCDomains}
  The domain $\mathcal D$ is {\it open} if $\mathcal D_n$ is open
  in $\gtupn$ for each $n$; it is {\it convex} if each $\mathcal D_n$
 is convex;
  and it is {\it matrix convex} if for any isometry
  $V:\mathbb R^{n_1} \to \mathbb R^{n_2}$ and
  $D=(X_1,\dots,X_g)\in\mathcal D_{n_2}$,
  $V^TDV=(V^TX_1V, \dots,V^T X_gV)\in \mathcal D_{n_1}$;
  the domain is {\it semi-algebraic} if there is a finite set
  $\mathcal P$ of symmetric polynomials such that
  $\mathcal D_n=\{X\in\gtupn: p(X)\succ 0, \mbox{ for all } p \in
 \mathcal P\}.$

  Below are several examples which are presented to illustrate
  the ideas or because they will play a role in the
  sequel.

 \begin{example}\rm {\bf The $\epsilon$-neighborhood of $0$.}
  Given $\epsilon$ the sequence of sets,
 \begin{equation*}
   N_n = \{ X\in\gtupn : \sum X_j^2 \prec \epsilon I_n\}
 \end{equation*}
  is an open matrix convex semi-algebraic domain.
 \end{example}

 \begin{example}{\bf Products.} \rm
   If $\mathcal U,\mathcal V \subset\gtup$  are domains, then
   so is the {\it product}
   $\mathcal U\times \mathcal V
   = (\mathcal U_n \times \mathcal V_n)_n \subset \ggtup$.

   The conditions, open, convex,  matrix-convex,  and semi-algebraic,
 are
   all preserved under products.
 \end{example}

  \begin{example}{\bf Coordinate Projections.} \rm
     Given a domain $\mathcal W \subset \ggtup$,
     let
  \begin{equation*}
    \pi_a(\mathcal W_n) = \{A\in\gatupn : \mbox{ there exists } X\in
 \gxtupn
      \mbox{ such that } (A,X) \in \mathcal W_n\}.
  \end{equation*}
    The coordinate projection $\pi_a$ preserves open, convex, and
 matrix
    convex domains.
 \end{example}

\subsubsection{Partial convex domains}
  Many of the notions surrounding domains have partial versions; i.e.,
  versions applied to a subset of the variables. For instance,
  a domain $\mathcal W$ in $\ggtup$
  is {\it open in $x$} if for each $n$ and $(A,X)\in\mathcal W_n$
 there
  is an open subset $U$ of $\gxtupn$ containing $X$
  such that $\{A\}\times U \subset \mathcal W_n$.

  In general, matrix convex implies convex (this depends upon
  the closed with respect to direct sums hypothesis).
  It turns out that if $\mathcal D$ is semi-algebraic, then
  convex implies matrix convex so that these two notions
  agree.

\subsubsection{Partial  Convexity}
  A polynomial $p(a,x)$ is {\it convex in $x$} on the open in $x$
 domain
  $\mathcal W$ if for each $n$ and $(A,X),(A,Y)\in \mathcal W_n$ such
   that $t(A,X)+(1-t)(A,Y) \in\mathcal W_n$ for each $0\le t\le 1$,
  it follows that
 \begin{equation*}
   t p(A,X) + tp(A,Y) \succeq p(A, tX+(1-t)Y).
 \end{equation*}

\subsubsection{Main Convexity Results}
 \label{subsubsec:main-convexity-results}
 This subsection contains most of the main results
 on partially convex polynomials.

\begin{theorem}
 \label{cor:convexinxforainU}
  Let $\mathcal W = (\mathcal  W_n)_n$ be an open domain
  in $\gtup$.
  If $p(a,x)$ is convex in $x$ on $\mathcal W$, then
$$
  p(a,x)= L(a,x) +V(a,x)^T Z(a) V(a,x),
$$
  where
 \begin{itemize}
   \item[(i)] $Z(a)$ is a (square) matrix-valued symmetric polynomial;
   \item[(ii)] $Z(A)\succeq 0$ for each $A\in \pi_a(\mathcal W)$;
   \item[(iii)] $V(a,x)$ is a (column) vector whose entries are linear
 in $x$; and
   \item[(iv)] $L(a,x)$ is a symmetric polynomials with degree at most
 one in $x$.
  \end{itemize}

  It follows that $p(a,x)$ is convex in $x$ on the product domain
   $\pi_a(\mathcal W) \times \mathbb S(\mathbb R^{g_x})$ and has
  degree at most two in $x$.
\end{theorem}

\begin{proof}See Corollary
\ref{cor:convexinxforainU-again}.\end{proof}

  If $Z(A)\succeq 0$ for all $A$, then, by a result in \cite{Mlaa}, $Z$
 factors as $Z=R^TR$ for a
  matrix-valued polynomial $R$ giving the following variant
  of Theorem \ref{cor:convexinxforainU}.

\begin{theorem}
 \label{cor:convexinxalla}
  Suppose $p(a,x)$
  is a symmetric polynomial and $\mathcal U=(\mathcal U_n)_n$
  is an open domain in $\mathbb S(\mathbb R^{g})$.
  If $p(a,x)$ is convex on the product domain
  $\gtup \times \mathcal U$,   then
$$
  p(a,x)= L(a,x) + \Lambda(a,x)^T \Lambda(a,x),
$$
   where
 \begin{itemize}
  \item[(i)] $\Lambda(a,x)$ is a (column) vector and is linear in $x$;
 and
  \item[(ii)] $L(a,x)$ has degree at most one in $x$.
 \end{itemize}

  Consequently, $p$ has degree at most two in $x$ and is globally
 convex
  in $x$.
\end{theorem}

  The previous theorem can be used to deduce the structure
  of polynomials  $p(a,x)$ which are either convex or concave in each
  variable separately.

\begin{theorem}
 \label{thm:concave/convex}
    Suppose $p(a,x)$ is symmetric. If $p$ is convex in $x$ and
    concave in $a$, then there exists linear (homogeneous
    of degree one) polynomials
    $r_j(x)$ and $s_l(a)$ and a polynomial $L(a,x)$ which
    has degree one in both $x$ and $a$ (so joint degree at most two)
    such that
 $$
   p(a,x) = L(a,x) + R(x)^T R(x) - S(a)^T S(a),
 $$
   where $R(x)$ is the (column) vector with entries $r_j(x)$
   and likewise for $S(a)$.
\end{theorem}

\begin{proof}See Theorem \ref{thm:concave/convex-again}.\end{proof}

\begin{theorem}
 \label{thm:convexinxanda}
 If $p(a,x)$ is (globally) convex in $a$ and $x$ separately, then there
 exists a
 polynomial $L(a,x)$ and and a (column)  vector of polynomials
 $\Lambda(a,x)$ which has degree at most one
 in $x$ and $a$ separately (thus at most degree two jointly) such that
$$
  p(a,x)= L(a,x)+\Lambda(a,x)^T \Lambda(a,x).
$$
  Here $\Lambda(a,x)$ is a column vector so that
  $\Lambda(a,x)^T \Lambda(a,x) = \sum \Lambda_j(a,x)^T \Lambda_j(a,x)$
  is a sum of squares.
\end{theorem}

\begin{proof}See Theorem \ref{thm:convexinxanda-again}. \end{proof}

\begin{remark}\rm
   The converse to each of the theorems in this subsubsection is
 evidently true.
\end{remark}

\subsection{Positivity of the Hessian}
 \label{subsec:positive-hessian}
  Just as in the classical commutative case, for a (non-commutative)
  symmetric polynomial, convexity implies that
  the {\it Hessian} - a version of
  the second derivative - is positive semi-definite.
  Conversely, a variety of fairly weak positivity hypotheses
  on the Hessian impose very strong restrictions on the polynomial.

\subsubsection{Definition of the Partial Hessian}
 \label{subsubsec:def-partial-hessian}
  Let $p(a,x)$ be a given symmetric polynomial. The {\it partial
 Hessian
  of $p$} with respect to $x$ in the direction $h=(h_1,\dots,h_{g})$
  is the formal second derivative of the polynomial in $t$,
  ${\tt{p}}(t)=p(a,x+th)$.  Thus, the partial Hessian,
 \begin{equation*}
   \frac{\partial^2 }{\partial x^2}p(a,x)[h]
     = {\tt{p}^{\prime\prime}}(0)
 \end{equation*}
  is a polynomial
  in $3g$ variables and is homogeneous of degree two in $h$.

  This partial Hessian can also be described algebraically
  on a monomial $m(a,x)$ by replacing each pair of variables
  $x_j$ and $x_k$ with $h_j$ and $h_k$ respectively and then
  multiplying by two. In particular, the partial Hessian of
  monomials of degree zero or one in $x$ is $0$; and
  the partial Hessian of a monomial of degree two in $x$
  is simply that monomial multiplied by $2$ with $x$ replaced by $h$.

 A simple observation, which will be used repeatedly
  without further comment, is that if $p(a,x)$ has degree two in $x$
  then $\Hxx p(a,x)[h]$ depends only on $x$ and
  $h$ and moreover there is a polynomial $L(a,x)$ of degree at most one
 in $x$
  so that
\begin{equation}
 \label{eq:deg-two}
   p(a,x)=\frac12 \Hxx p(a,x)[x] + L(a,x).
\end{equation}
  Moreover, this representation is unique in the sense that
  the Hessian is the part of $p(a,x)$ which is homogeneous
  of degree two in $x$ and $L(a,x)$ contains the remaining
  part of $p$.  In particular, there can be no cancellation between
  these two terms.

  Examples will be given in Section \ref{subsec:examples}.

\subsubsection{Convexity and the Hessian}
 \label{subsubsec:convexity-hessian}
  The following Proposition
  says that  convexity of $p$ on an open in $x$ domain $\mathcal D$
  implies positivity of the partial Hessian
  on $\mathcal D$.

\begin{proposition}
 \label{prop:convex-hessian}
    Suppose $p(a,x)$ is a symmetric polynomial,  $A\in\gatupn,$
    and $\mathcal U$ is an open convex set in $\gxtupn$. If
    $p(A,X)$ is convex for $X\in \mathcal U$, then, for
    each $X\in\mathcal U$ and $H\in\gxtupn,$
 \begin{equation*}
     \Hxx p(A,X)[H]\succeq 0.
 \end{equation*}

   Explicitly, the hypothesis on $p$ is that for the given
   $A$ and each $X,Y\in\mathcal U$, and $0<t<1$,
 \begin{equation*}
   p(A,tX+(1-t)Y) \preceq   t p(A,X)+ (1-t) p(A,Y).
 \end{equation*}
\end{proposition}

\begin{proof}
  Fix $X\in\mathcal U$.
  For a given $H\in\gatupn$ and $s\in\mathbb R$ small,
   $X\pm sH\in\mathcal U$. By the convexity hypothesis on $p$,
 \begin{equation*}
   0\preceq (p(A,X+sH)+p(A,X-sH))-2p(A,X) = s^2 \Hxx p(A,X)[H] +
 o(s^4).
 \end{equation*}
  Dividing by $s^2$ and letting $s$ tend to $0$ gives the result.
\end{proof}

  Positivity of the Hessian not only implies
  convexity of $p$, but also imposes further
  serious constraints.

\begin{theorem}
 \label{prop:irr-pos-Hessian}
  Suppose $p(a,x)$ is a symmetric polynomial
  of degree
  $d_a$ and $d_x\ge 2$ in $a$ and $x$ respectively
  and $\mathcal W \subset \ggtup$
  is a domain.   If
 \begin{itemize}
  \item[(i)] $\mathcal W$ is open in $a$; and
  \item[(ii)] for each $n$ and each  $(A,X)\in\mathcal W_n$ and every
    $H\in\gxtupn$,
   \begin{equation*}
      \Hxx p(A,X)[H] \succeq 0,
   \end{equation*}
 \end{itemize}
    then, {\bf either}
 \begin{itemize}
  \item[(A)] there is a nonzero polynomial $q$ of degree
    at most $d_a$ and $d_x-2$ in $a$ and $x$
    respectively, so that $q(A,X)=0$ on $\mathcal W$;  {\bf or}
  \item[(B)] $p$ has degree at most two in $x$, and for each
 $A\in\pi_a(\mathcal W_n)$,
       the function $p(A,X)$ is (globally) convex in $X\in\gtupn.$
 \end{itemize}

    In particular, if $\mathcal W$ is the product
    $\mathcal W =  \gtup \times \mathcal V$,
    for some matrix convex set $\mathcal V=(\mathcal V_n)_n$, then
    {\bf either}
  \begin{itemize}
   \item[($A^\prime$)] there is a polynomial $q$ of degree at most
    $d_a$ and $d_x-2$ in $a$ and $x$ respectively
     so that $q(A,X)=0$ for all $(A,X)\in\mathcal W$;  {\bf or}
   \item[($B^\prime$)] there exists $\Lambda(a,x)$, linear in
     $x$, and $L(a,x)$ a polynomial of degree at most one in $x$ so
 that
    $$
      p(a,x)= L(a,x)+\Lambda(a,x)^T \Lambda(a,x).
    $$
  \end{itemize}
\end{theorem}

\begin{proof}See Section \ref{sec:partial-pos}. \end{proof}

  Proposition \ref{thm:local-again}, which isolates
  the role of positivity of the Hessian, is a not too technical,
  but still widely applicable, ingredient
  in the proof of Theorem \ref{prop:irr-pos-Hessian}.

\subsection{Further Results and Organization of the Paper}
    The proofs of the main results turn on
    the {\it border vector-middle matrix}
    representation of the partial Hessian.
    This representation and some
    examples   are discussed
    in the next section, Section \ref{sec:middlematrix}.

    The proofs of Proposition \ref{thm:local-again},
    Theorem \ref{prop:irr-pos-Hessian} and Corollaries
    thereof are in Section \ref{sec:partial-pos}.
    The proofs of the results in Subsubsection
    \ref{subsubsec:main-convexity-results}
   are in Sections \ref{sec:separate-convexity}
   and \ref{sec:concave/convex}.

    A structure result
    for the middle matrix of the partial
    Hessian which is of independent
    interest and will likely be a valuable tool
    in further investigations is presented
    in Section \ref{sec:poly-congruence}
    which can be read following Section \ref{sec:middlematrix}.

\section{The Middle Matrix of the Partial Hessian}
 \label{sec:middlematrix}
  A polynomial $q(a,x)[h]$ in the non-commuting
  variables $a,x$ and $h$ which is homogeneous of
  degree two in $h$ is conveniently represented in
  terms of a Gram-type representation in the form
$$
  q(a,x)[h] = V(a,x)[h]^T Z(a,x)V(a,x)[h],
$$
  where $Z(a,x)$ is a symmetric matrix with polynomial
  entries, and $V(a,x)[h]$ is a vector, called the \emph{border vector},
   whose entries
  have the form $h_j m(a,x)$ for monomials $m$ of
  degree at most the degree of $q$ in $(a,x)$.  Thus, we can index
  the entries of $Z := Z(a,x)$ by the monomials $h_j m(a,x)$.   The
  notation $Z(\cdot , \cdot)$ will be used to denote the matrix
  $Z$, as well as the entries of $Z$.  For example,
  $Z(h_j m(a,x),h_k m'(a,x))$ denotes the entry of $Z$, corresponding
   to entries $h_j m(a,x)$ and $h_k m'(a,x)$ of the border vector.
   Context will make clear which of the two uses of the
   $Z(\cdot,\cdot)$ notation is being employed.

  The matrix $Z=Z(a,x)$ is known as the {\it middle matrix}
  and is unique up to order and presence of zero rows and columns. When
  $q$ is symmetric, so is $Z(a,x)$.

  Of special interest is the case that $p(a,x)$ is
  symmetric and $q(a,x)[h]=\Hxx p(a,x)[h]$. A central
  object in the analysis to follow is the {\it derived}
  middle matrix,
$$
  \mathcal Z(a)= Z(a,0).
$$

Suppose $p(a,x)$ is a polynomial of degree $d_x$ in $x$ and $d_a$ in
$a$.  Let ${Z}$ be the middle matrix for the partial Hessian of $p$.
We take the border vector for ${Z}$ to be of the form
$$V = \left[ \begin{array}{c}V_0 \\
\vdots \\
V_{d_x - 2} \end{array} \right],$$ where $V_j$ is a vector whose
entries consist of words of the form
\begin{equation*}
 \label{eq:borderterm}
  h_{*}m_0x_{k_{1}}m_{1} \cdots m_{j-1}x_{k_{j}}m_{j},
\end{equation*}
 where $m_{0}, \dots ,m_{j}$ are words in $a$ such that $\mbox{deg}_a
\ m_0+ \cdots + \ \mbox{deg}_a \ m_j \le d_a.$  That is,
 $V_j$ captures those monomials in $x$
 and $a$ of total degree at most $d_a$ in
 $a$ and degree exactly $j$ in $x$. Thus $V_j$ is a vector of height
$$
 N_j := g_x^{j+1} \left( \sum_{n_0+ \cdots +n_{j+1} \le d_a} g_a^{n_0+
 \cdots
 +n_{j+1}} \right)
$$
and $Z$ has $(d_x-2)(N_0 + \cdots + N_{d_x-2})$ rows and columns.
With respect to this block form of $V$, we can write $Z$ in block
form $Z = [Z_{ij}]$, for $i,j = 1, \dots , d_x-2$. Thus, $Z_{ij}$ is
that part of $Z$ which corresponds to terms the $V_i$ part of the
border vector on the left and the $V^j$ part on the right.

\subsection{A few simple examples}\label{subsec:examples}

\begin{example}
\label{ex:ax^3a}

Let $p = ax^3a$.  Then $q(a,x)[h] = 2[ahxha + ah^2xa + axh^2a]$,
which is equal to

$$\left[
2\begin{array}{ccccc} h & ah & xh & axh & xah \end{array} \right]
\left[ \begin{array}{ccccc}
0 & 0 & 0 & 0 & 0 \\
0 & x & 0 & 1 & 0 \\
0 & 0 & 0 & 0 & 0 \\
0 & 1 & 0 & 0 & 0 \\
0 & 0 & 0 & 0 & 0 \end{array}\right] \left[ \begin{array}{c}
h \\
ha \\
hx \\
hxa \\
hax \end{array} \right] $$

\end{example}

\begin{example}\label{ex:ax^3plus}
The polynomial $p = ax^3 + x^3a$ has Hessian $2[ahxh + ah^2x + axh^2
+ hxha + h^2xa + xh^2]$ which has the representation

$$\left[
\begin{array}{ccccc}
h & ah & xh & axh & xah
\end{array}
\right] 2 \left[
\begin{array}{ccccc}
0 & x & 0 & 1 & 0 \\
x & 0 & 1 & 0 & 0 \\
0 & 1 & 0 & 0 & 0 \\
1 & 0 & 0 & 0 & 0 \\
0 & 0 & 0 & 0 & 0%
\end{array}
\right] \left[
\begin{array}{c}
h \\
ha \\
hx \\
hxa \\
hax%
\end{array}
\right] $$
\end{example}

\begin{example}\label{ex:x^2axplus}
Let $p= x^2ax + xax^2$.  Then $q(a,x)[h] =2 [hxah + h^2ax+ xhah +
haxh + xah^2 + hahx]$, which equals

$$\left[
\begin{array}{ccccc} h & ah & xh & axh & xah \end{array} \right]
2 \left[ \begin{array}{ccccc}
xa + ax & 0 & a & 0 & 1 \\
0 & 0 & 0 & 0 & 0 \\
a & 0 & 0 & 0 & 0 \\
0 & 0 & 0 & 0 & 0 \\
1 & 0 & 0 & 0 & 0 \end{array}\right] \left[ \begin{array}{c}
h \\
ha \\
hx \\
hxa \\
hax \end{array} \right] $$
\end{example}

Recall that $Z_{00}$ is that part of the middle
matrix which corresponds to those terms of the
border vector with no $x$ variables (both sides). Thus,
in this last example we have
$$\
{Z}_{00} = 2 \left[ \begin{array}{cc}
xa + ax & 0 \\
0 & 0 \end{array}\right],
$$
and likewise
$$
{Z}_{01} = 2 \left[ \begin{array}{ccc}
a & 0 & 1 \\
0 & 0 & 0 \end{array}\right],
$$
and the derived matrix is
$$\mathcal{Z} = \mathcal{Z}(0,a) = 2
\left[ \begin{array}{ccccc}
0 & 0 & a & 0 & 1 \\
0 & 0 & 0 & 0 & 0 \\
a & 0 & 0 & 0 & 0 \\
0 & 0 & 0 & 0 & 0 \\
1 & 0 & 0 & 0 & 0 \end{array}\right]
$$

\subsection{Faithfulness}
  The following standard Lemma says more than  the totality of matrix
  evaluations are faithful on $\Rax$.  Given $g$ and $d$, let
 \begin{equation}
  \label{eq:defNgd}
    N(g,d)=\sum_0^d g^j.
 \end{equation}

 \begin{lemma}
  \label{lem:weakfaith}
    Suppose $p(y)$ is a polynomial of degree $d$ in $g$ variables.
    If $p$ vanishes on an open subset $U\subset \mathbb
 S_{N(g,d)}(\mathbb R^g)$,
    then $p=0$.
 \end{lemma}

 \begin{proof} See Lemma 2.2 in \cite{HM04}.
\end{proof}

\section{The Structure of Partially Convex Polynomials}
 \label{sec:partial-pos}
 In this section we consider positivity conditions on the partial
 Hessian of a symmetric polynomial $p$ which impose strong constraints
 on the form of $p$.

 We begin with a result which isolates the role of direct
  sums and positivity of the Hessian.

 \begin{proposition}
  \label{thm:local-again}
   Let $p=p(a,x)$ be a given symmetric polynomial of degree
   $d_a$ and $d_x$ in $a$ and $x$ respectively.
   Suppose $n\ge N(g,d)$, $v\in\mathbb R^n$, and that
      $U$ is an open set in $\gtupn.$
   If
 \begin{itemize}
    \item[(i)] the set   $\{m(A,\chi)v: \degx m \le d_x-2 \and \dega m
 \le d_a\}$
     is linearly independent in $\mathbb R^{n}$
      for each $A\in U$; and
    \item[(ii)] for each natural number $\ell$ and $A\in U$,
     $$
       0\le   \langle \Hxx p(A\otimes I_\ell,\chi\otimes
 I_\ell)[H]v,v\rangle
     $$
      for all $H\in\mathcal S_{n\ell}(\mathbb R^{g_x})$,
  \end{itemize}
    then
 \begin{itemize}
  \item[(a)] the degree of $p$ in $x$ is at most  two;
  \item[(b)] the partial Hessian of $p$  takes the form
   $$
    \Hxx p(a,x)[h]= (V(a)[h])^T Z_{0,0}(a) V(a)[h];
   $$
  \item[(c)] $Z_{0,0}(A)\succeq 0$ for $A\in U$; and
  \item[(d)] there is a symmetric polynomial  $L(a,x)$ which
    has degree at most one in $x$ so that
   $$
     p(a,x) = \frac12 V(a)[x]^T Z_{0,0}(a)V(a)[x] + L(a,x).
   $$
    In particular, $p(A,X)$ is (globally) convex in $X\in\gxtupn$ for
 each fixed
   $A\in U$.
  \end{itemize}
 \end{proposition}

\def\tA{\tilde{A}}
\def\tchi{\tilde{\chi}}
\def\tv{\tilde{v}}

\begin{proof}
  The partial Hessian can be represented in terms of the middle matrix
  $Z(a,x)$   as
$$
  \Hxx p(a,x)[h]= V(a,x)[h]^T Z(a,x) V(a,x)[h].
$$

  Given $\ell$, let $[v]_\ell$ denote the vector in $\mathbb
 R^{n\ell}=(\mathbb R^n)^\ell$
  with each of its $\ell$ block (of length $n$) entries equal to $v$.
  Since the set $\{m(A\otimes I_\ell,\chi\otimes I_\ell)[v]_\ell: \degx
 m \le d_x-2 \and \dega m \le d_a \}$ is linearly independent, the
  \cite{CHSY} Lemma implies that the subspace $\{V(A\otimes
 I_\ell,\chi\otimes I_\ell)[H][v]_\ell:H\in \mathcal S_{n\ell}(\mathbb R^{g_x})\}$
  has codimension at most $\kappa$ in $\mathbb R^{nN(g,d)}$
 independent
  of $n$. (See the appendix, \S \ref{appendix:CHSY}
  for the details.) Choose $\ell$ so that $\kappa < \ell.$
  With this fixed $\ell$, to simplify the notations, let
  $\tA=A\otimes I_\ell$; $\tchi=\chi\otimes I_\ell$; and
 $\tv=[v]_\ell$.

  It follows that $Z(\tA,\tchi)$ has at most $\kappa<\ell$ negative
 eigenvalues
  for each $A\in U$.
  We now partition $V(a,x)[h]$ and $Z(a,x)$ in blocks by
  the degree of $x$ and let $Z_{jl}(a,x)$ denote the part
  of $Z(a,x)$ corresponding to the terms of degree $j$ and
  $l$ in $x$. In particular, $Z_{0,d_x-2}(a,x)$ represents
  those terms in the Hessian of the form $m_l(a)h_*
 s(a,x)h_*m_r(a,x)$,
  where $m_r(a,x)$ has degree $d_x-2$ in $x$ and where $m_r,m_l,s$
  are monomials. (Here we are assuming the degree in $x$ is $d_x$,
  if it is lower, make the obvious adjustments). It follows that
  $Z_{0,d_x-2}(a,x)=Z_{0,d_x-2}(a)$ is independent of $x$.

  Now $Z_{0,d_x-2}(a)$ is a matrix with  polynomial in $a$ entries.
  Since $n\ge N(g,d)$ and $U$ is an open subset of $\mathbb
 S_{N(g,d)}(\mathbb R^{g_a})$,
  it follows from Lemma \ref{lem:weakfaith} that $Z_{0,d_x-2}(A)\ne 0$
  for some (and hence most) $A\in U$.  Thus, $Z_{0,d_x-2}(\tA)$
  has rank at least $\ell$.  Hence,  from the form of
  $Z$, if $d_x-2>0$, then $Z(A,\chi)$ has at least $\ell$
  negative eigenvalues, contradicting $\kappa<\ell$. We conclude
  that $d_x\le 2$.

  Now that we know the degree of $p$ is two in $x$, the middle
  matrix $Z(a,x)=Z_{0,0}(a)$.  Which gives the advertised
  representation.

  For the last part of the theorem, we argue as above and find
  that $Z(\tA)=Z_{0,0}(\tA)= Z(A)\otimes I_\ell$
  has at most $\kappa <\ell$ negative eigenvalues. From tensor
  product form it follows that $Z(A)$ can have no negative eigenvalues
  as otherwise $Z(\tA)$ has at least $\ell$.   Thus $Z(A)\succeq 0$.

  Since $Z(A)\succeq 0$ and (since $p$ has degree at most two in $x$),
$$
 p(a,x)= V(a)[x]^T Z(a) V(a)[x] + L(a,x)
$$
 it follows that, with $A$ fixed, that $p(a,x)$ is convex
  in $X\in\gxtupn$.
\end{proof}

\begin{remark}\rm
  As seen in the proof, in item (ii) the for every $\ell$ can
  be replaced with some $\ell\ge\kappa + 1$ where $\kappa$ is an
  integer which depends only upon $d$ and $g$.
\end{remark}

In view of the remark and the following
Lemma, Hypothesis (ii) of Theorem \ref{thm:local-again}
is  stronger than needed in that it suffices
to tensor with $I_\ell$ for certain $\ell$,
thereby weakening the need to  tensor with $I_\ell$ at all.

\begin{lemma} (R. Guralnick and L. Small)
 \label{lem:G-S}
  Let $q\ne 0$ be a given (not necessarily symmetric)
  polynomial of degree $d$ in $g$ variables.
  Then there is a sequence of integers $n_k \to \infty$
  for which
 \begin{equation*}
   \{A\in\mathbb S_{n_k}(\mathbb R^g):
   q(A) \mbox{ has full rank }\}
 \end{equation*}
  is open and dense in $\mathbb S_{n_k}(\mathbb R^g).$
\end{lemma}

See Appendix \ref{sec:GS} for a proof.
Likely, the special sequence $n_k$ can be replaced
by: any $n\ge n_0$ for some
 $n_0$ depending only upon
  $g$ and $d$.

  The proof of Theorem \ref{prop:irr-pos-Hessian}
  combines Proposition \ref{thm:local-again} and
  the following two lemmas.

\begin{lemma}
 \label{lem:ind-irr-directsum}
   Suppose $\mathcal W\subset \ggtupn$ is a domain and $d_a$
   and $d_x$ are natural numbers. Either there exists
   an $n$ and $(A,X)\in\mathcal W_n$ and $v\in\mathbb R^n$ such that
 the set
   $\{m(A,X)v : m {\mbox{ is a monomial with }} \degx m \le d_x-2, \
 \dega m \le d_a\}$
   is linear independent, or there is a polynomial $q$ (not
 necessarily
   symmetric) of degree at most $d_a$ and $d_x$ in $a$ and $x$
 respectively
   so that $q(A,X)=0$ for all $(A,X)\in\mathcal W$.
\end{lemma}

\begin{proof}
  A proof can be found in \cite{DHMiumj}[Lemma 4.1].
\end{proof}

 This Lemma naturally combines with the following simple observation.

\begin{lemma}
 \label{lem:for-free}
   Suppose $\mathcal W \subset \ggtupn$ is a domain and
   and natural numbers $d_a$ and $d_x$ are given. If there
   exists $n$ and $(A,X)\in\mathcal W$ and $v\in\mathbb R^n$
   such that
   $\{m(A,X)v : m {\mbox{ is a monomial with }} \degx m \le d_x-2, \
 \dega m \le d_a\}$
   is linearly independent, then for each $(B,Y)\in\mathcal W$
   and vector $w$ (of the correct size), the set
  $\{m(A\oplus B,X\oplus Y)(v\oplus w)
    : m {\mbox{ is a monomial with }} \degx m \le d_x-2, \ \dega m \le
 d_a\}$
  is linearly independent.
\end{lemma}

\begin{proof}[Proof of Proposition \ref{prop:irr-pos-Hessian}]
   From Lemma \ref{lem:ind-irr-directsum} either there is
   a (not necessarily symmetric) polynomial $q$
   of degree at most $d_a$ and $d_x-2$ in
   $a$ and $x$ respectively  so that $q(A,X)=0$ on $\mathcal W$, or
   there is an $n$, a pair $(B,Y)\in\mathcal W_n$
   and vector $u\in\mathbb R^n$ so that
   the set $\{m(B,Y)u: \degx m \le d_x-2 \and \dega m \le d_a\}$ is
 linearly independent.

   Given $(C,Z)\in\mathcal W,$
   let $(A^\prime,\chi)=(B\oplus C, Y\oplus X)\in\mathcal W$ and
 $v=0\oplus u$.
   By Lemma \ref{lem:for-free}, the set
   $\{m(A^\prime,X)v: \degx m \le d_x-2 \and \dega m \le d_a\}$ is
 linearly independent.
   In particular, by the open in $a$ hypothesis on
   $\mathcal W$, there is a neighborhood $U$ of $A^\prime$ such that
   for all $A\in U$ the set
   $\{m(A,X)v: \degx m \le d_x-2 \and \dega m \le d_a\}$ is linearly
 independent.
   Since $\mathcal W$ is also a domain on which the Hessian is
 non-negative,
   Theorem \ref{thm:local-again} applies with the conclusion
   that $p$ has degree at most two in $x$ and
 \begin{equation*}
   p(a,x) = \frac12 V(a)[x]^T Z(a) V(a)[x] + L(a,x),
 \end{equation*}
   where $L(a,x)$ has degree at most one in $x$, and
  $Z(A)\succeq 0$ for $A\in U$.  In particular, $Z(B\oplus C)\succeq
 0.$
   Therefore $Z(C)\succeq 0$ and $p(C,X)$ is convex
   (globally) in $X$  for each $C\in \pi_a(\mathcal W)$.

   For the second part of the corollary, note that in the
   absence of the polynomial $q$, the first part of the
   corollary implies that $Z(A)\succeq 0$ for all $A$.
   Thus, by a Theorem in \cite{Mlaa}, $Z(A)$ factors
   as a SoS and the result follows.
\end{proof}

 We close this section by pointing out the following
 simple special cases of Proposition \ref{prop:irr-pos-Hessian}.

\begin{corollary}
 \label{cor:convexinxforainU-again}
  If $p(a,x)$ is convex in $x$ on an open domain $\mathcal W \subset
 \ggtup$, then
$$
  p(a,x)= L(a,x) +V(a,x)^T Z(a) V(a,x),
$$
  where
 \begin{itemize}
   \item[(i)] $Z(A)\succeq 0$ for $A\in\pi_a(\mathcal W)$;
   \item[(ii)] $V(a,x)$ is the border vector, which is linear  in $x$;
 and
   \item[(iii)] $L(a,x)$ has degree at most one in $x$.
  \end{itemize}
\end{corollary}

\begin{proof}
   Note that Lemma \ref{lem:weakfaith} rules out the
   possibility that there is a nonzero polynomial $q$
   such that $q(A,X)=0$ for all $(A,X)\in\mathcal W$.
   The convexity implies that $\Hxx p(A,X) \succeq 0$ for
   all $(A,X)$ in the open domain $\mathcal W$.
   Consequently, the conclusion of the Lemma follows
   from the argument given for the proof of Proposition
   \ref{prop:irr-pos-Hessian}.
\end{proof}

\begin{corollary}
 \label{cor:convexinxalla-again}
  Suppose $\mathcal V\subset \gtup$
  is an open matrix convex set and
  let $\mathcal W= \gtup  \times \mathcal V$. If
  $p(a,x)$ is convex in $x$ on $\mathcal W$,
  then
$$
  p(a,x)= L(a,x) + \Lambda(a,x)^T \Lambda(a,x),
$$
   where
 \begin{itemize}
  \item[(i)] $\Lambda(a,x)$ is a vector and is linear in $x$; and
  \item[(ii)] $L(a,x)$ has degree at most one in $x$.
 \end{itemize}
\end{corollary}

\begin{proof}
  As in the previous corollary, there does not exist
  a nonzero polynomial  $q$ such that $q(A,X)=0$ for
  all $(A,X)\in\mathcal W$. Thus, option $(B^\prime)$
  of Proposition \ref{prop:irr-pos-Hessian} occurs.
\end{proof}

\section{Separate Convexity}
 \label{sec:separate-convexity}
\begin{theorem}
 \label{thm:convexinxanda-again}
 If $p$ is (globally) convex in $a$ and $x$ separately, then there
 exists
 an $m$ and polynomials $L(a,x)$ and  $\Lambda_j(a,x)$,
 $j=1,2,\dots,m$,
 which are degree (at most) one
 in $x$ and $a$ separately (thus at most degree two jointly) such that
$$
  p(a,x)= L(a,x)+\Lambda(a,x)^T \Lambda(a,x).
$$
  Here we have used the shorthand,
  $\Lambda(a,x)^T \Lambda(a,x) = \sum \Lambda_j(a,x)^T
 \Lambda_j(a,x).$
\end{theorem}

  We begin with a lemma.

\begin{lemma}
 \label{lem:convex/degree2}
    Suppose $p(a,x)=q(a,x)+\sum r_j(a,x)^T r_j(a,x)$.
    If
  \begin{itemize}
   \item[(i)] each $r_j(a,x)$ is homogeneous of degree one in $x$;
   \item[(ii)] $q(a,x)$ is degree one in $x$;  and
   \item[(iii)] $p(a,x)$ has degree at most two in $a$,
  \end{itemize}
    then each
    $r_j(a,x)$ has degree at most one in $a$ and $q(a,x)$ has
    degree at most two in $a$.
\end{lemma}

\begin{proof}
  Terms from $q(a,x)$ cannot cancel those from
  $s(a,x)=\sum r_j(a,x)^Tr_j(a,x)$, since the former are
  of at most degree one in $x$ and the later
  homogeneous of degree two in $x$.
  Since $p$ has degree two in $a$ and there can't be cancellation
   of the highest degree term
  in $a$ in the sum of squares term $s$, each $r_j(a,x)$ has
  degree at most one in $a$. Likewise,
  $q(a,x)$ has degree at most two in $a$.
\end{proof}

  Let $\mathcal J$ denote those monomials in $x$ and $a$ which are
 linear
  in each of $x$ and $a$ separately.  Thus $\mathcal J$ has
  $2g^2$ elements.
  Let $V(a,x)$ denote the tautological vector whose entries
  are the monomials from $\mathcal J$.
  Given a $\mathcal J\times \mathcal J$ matrix $M$, the expression,
 \begin{equation}
  \label{eq:s-middle}
    s(a,x)=V(a,x)^T M V(a,x) =\sum_{m,\ell} M_{m,\ell} m(a,x)\ell(a,x)
 \end{equation}
  is then a polynomial which is homogeneous of degree two in
  each of $x$ and $a$ separately.

  If $M=R^T R$, then
\begin{equation}
 \label{eq:s-sos}
    s(a,x) = \sum r_j(a,x)^T r_j(a,x)
\end{equation}
  where $r_j(a,x)= R_j V(a,x)$
  and $R_j$ is the $j$-th row of $R$.

  Conversely, if $s$ has the form in equation (\ref{eq:s-middle}),
  and $r_j =\sum r_j(m) m(a,x)$ (where the sum is over $m\in\mathcal
 J$),
  then $s$ has the form in equation (\ref{eq:s-middle} with
  $M=R^T R,$ where $R$ is the matrix whose $j$-th row has entries
  $R_j(m)=r_j(m)$.

\def\pA{A^\prime}

\begin{lemma}
 \label{lem:Aunique}
   Suppose $s(a,x)$ is homogeneous of degree two in each of $x$
   and $a$ separately.  If $A$ and $A^\prime$ are $n\times \mathcal J$
   and $n^\prime \times \mathcal J$ matrices respectively such that
 \begin{equation*}
   s(a,x) = V(a,x)^T A^T A V(a,x) = V(a,x)^T (\pA)^T \pA V(a,x),
 \end{equation*}
   then there is a partial isometry $U:\mathbb R^n \to \mathbb  R^{n^\prime}$
   such that $A=U\pA$.
\end{lemma}

\begin{proof}
  It is readily verified that
 \begin{equation}
  \label{eq:ATA}
     A^TA= (\pA)^T \pA.
 \end{equation}
   The existence
  of $U$ now follows from the Douglas Lemma.  For convenience
  of the reader we give the argument in this case.  Define
  $U$ on the range of $A$ into the range of $\pA$ by
  $UAm= \pA m$ (here $m$ is the vector with a $1$ in the $m$-th
  place and $0$ elsewhere). The equality \eqref{eq:ATA} implies
  $U$ unitary and thus extends to a partial isometry (by defining
  it to be $0$ on the orthogonal complement of the range of $A$).
 \end{proof}

   Similarly, a polynomial $p(x)$ which homogeneous of degree two in
 $x$ alone
   (no $a$) can be represented as
 \begin{equation*}
     p(x)= V(x)^T P V(x),
 \end{equation*}
  where $V(x)$ is now the vector with entries $x_j$.  The polynomial
 $p$
  is a sum of squares if and only if $P=R^TR$ for some $R$ (so if and
 only
  if $P$ is positive semi-definite).  In particular, we are using $V$
  in two different ways which should cause no confusion.

\begin{proof}[Proof of Theorem \ref{thm:convexinxanda-again}]
  Since $p$ is convex in $x$, Corollary \ref{cor:convexinxalla}
 implies
  $p$ can be written in the form,
$$
   p(a,x)= {\tt L}(a,x)+ {\tt H}(a,x)^T {\tt H}(a,x),
$$
  where ${\tt H}(a,x)$ is linear in $x$ and ${\tt L}(a,x)$ has degree
 at most
  one in $x$.  Here ${\tt H}$ is a vector with entries ${\tt H_j}$.

  Since $p$ is convex in $a$ it has degree at most two in $a$. Thus,
  Lemma \ref{lem:convex/degree2} says ${\tt H}(a,x)$ has degree
  at most one in $a$ and ${\tt L}(a,x)$ has degree at most two in $a$,
  in addition to the degree restrictions relative to $x$ above.

  Write ${\tt H_j}(a,x)= a_j(a,x)+b_j(x)$ with $a_j(a,x)$ homogeneous
  of degree one in $a$ and $b_j(x)$ a polynomial in $x$ alone.
  Similarly, since ${\tt L}$ has degree at most one in $x$ and
  two in $a$, it can be written as
$$
 {\tt L}(a,x)= C(a,x)+D(a)+E(a,x)
$$
  where $C(a,x)$ is homogeneous of degree two in $a$ and one in $x$;
  $D(a,x)$ is homogeneous of degree two in $a$ (and has no $x$);
  and $E$ has degree at most one in each of $x$ and $a$.

  Now let $A$ and $B$ respectively denote matrices  which produce
  the representations
\begin{equation*}
  \begin{split}
   \sum a_j^T a_j = & V(a,x)^T A^T A V(a,x) \\
   \sum b_j^T b_j =& V(x) B^T B V(x).
 \end{split}
\end{equation*}

  We have,
\begin{equation*}
 \begin{split}
   p(a,x) = &  V(a,x)^T A^TA V(a,x) + [V(a,x)A^T B V(x) + V(x)^T B^T A
 V(a,x)]\\
    &  + C(a,x)
    + V(x)B^T B V(x) +D(a)+E(a,x).
 \end{split}
\end{equation*}
  Note that the term $[\cdot ]$ is the part of $p(a,x)$ which is
 homogeneous of degree two in $x$ and one in $a$;
  whereas $C(a,x)$ is the part homogeneous of degree one in $x$ and two
 in $a$.

\def\pB{B^\prime}
\def\pC{C^\prime}
\def\pD{D^\prime}
\def\pE{E^\prime}

 Reversing the roles of $x$ and $a$, $p$ can also be written as
\begin{equation*}
 \begin{split}
   p(a,x) =  & V(a,x)^T (\pA)^T\pA V(a,x) +
    [V(a,x)(\pA)^T \pB V(a) + V(a)^T (\pB)^T  \pA V(a,x)] \\
    & + \pC(a,x)
     + V(a)(\pB)^T \pB V(a) +\pD(x)+\pE(a,x).
 \end{split}
\end{equation*}
  Note that the term $[\cdot ]$ is the part of $p(a,x)$ which is
 homogeneous of degree two in $a$ and one in $x$;
  whereas $\pC(a,x)$ is the part homogeneous of degree one in $a$ and
 two in $x$.

  Comparing these last two representations we find,
\begin{equation*}
 \begin{split}
   V(a,x)A^T A V(a,x) = & V(a,x) (\pA)^T \pA V(a,x) \\
  V(a,x)A^T B V(x) + V(x)^T B^T A V(a,x) = & \pC(a,x) \\
  V(a,x)(\pA)^T \pB V(a) + V(a)^T (\pB)^T \pA V(a,x) = & C(a,x)\\
  V(a)^T(\pB)^T  \pB V(a)=&D(a).
 \end{split}
\end{equation*}

  From Lemma \ref{lem:Aunique} there is a partial isometry $U$ so that
  $\pA = UA$.  Choose $W$ so that $I-UU^T =W W^T$. Consider,
\begin{equation*}
 \begin{split}
  (AV(a,x)&+BV(x)+U^T\pB V(a))^T  (AV(a,x)+BV(x)+U^T\pB(a)V(a))\\
      &+ (W\pB V(a))^T (W\pB V(a)) \\
        & - [V(a)^T(\pB)^T UB V(x)+ V(x)^TB^T U^T \pB V(a)] + E(a,x)
 \\
    =\ \ & (AV(a,x)+BV(x))^T (AV(a,x)+BV(x)) \\
      & + V(a,x)^T A^T U^T\pB V(a)
              + V(a)^T (\pB)^T UAV(a,x)\\
      &  + V(a)^T B^T UU^T B V(a) + V(a)^T (\pB)^T W^T W \pB V(a) \\
    =\ \ & (AV(a,x)+BV(x))^T (AV(a,x)+BV(x)) \\
       & + [V(a,x)^T (\pA)^T \pB V(a) + V(a)^T (\pB)^T \pA V(a,x)] \\
       &+ V(a)^T B^T B V(a) +E(a,x)\\
    =\ \ & (AV(a,x)+BV(x))^T (AV(a,x)+BV(x)) + C(a,x)+ D(a)+E(a,x) \\
    =\ \ & p(a,x).
 \end{split}
\end{equation*}
\end{proof}

\section{Convex in $x$ and Concave in $a$}
 \label{sec:concave/convex}

\begin{theorem}
 \label{thm:concave/convex-again}
    Suppose $p(a,x)$ is symmetric. If $p$ is convex in $x$ and
    concave in $a$, then there exists linear (homogeneous
    of degree one) polynomials
    $r_j(x)$ and $s_l(a)$ and a polynomial $L(a,x)$ which
    has degree one in both $x$ and $a$ (so joint degree at most two)
    such that
 $$
   p(a,x) = L(a,x) + R(x)^T R(x) - S(a)^T S(a),
 $$
   where $R(x)^T R(x)$ and $S(a)^TS(a)$ are shorthand for the
   sums $\sum r_j(x)^Tr_j(x)$ and $\sum s_j(a)^T s_j(a)$ respectively.
\end{theorem}

 \begin{proof}
    Since $p$ is (globally) convex in $x$, it can be written in the
 form,
$$
  p(a,x) = \Lambda(a,x)^T \Lambda(a,x) + L(a,x),
$$
  where $\Lambda$ is linear in $x$ and $L(a,x)$ has degree at most
  one in $x$.  Thus,  $\Lambda(a,x)^T\Lambda(a,x)$ is homogeneous
  of degree two in $x$; whereas $L(a,x)$ had degree at most one in
 $x$.
  In particular, there can be no cancellation between these terms.

   Since $p$ is (globally) concave in $a$, it has degree at most
   two in $a$ and since the terms in $\Lambda(a,x)^T \Lambda(a,x)$
   can not cancel with those in $L(a,x)$, it thus follows that
   $\Lambda(a,x)$ has degree at most one in $a$ and likewise
   $L(a,x)$ has degree at most two in $a$.

   Write,
$$
   L(a,x) = L_0(x)+ L_1(a,x) + L_2(a,x),
$$
  where $L_j$ is homogeneous of degree $j$ in $a$ (and
  degree at most one in $x$).  Similarly, write
   $\Lambda(a,x) = \Lambda_0(x) + \Lambda_1(a,x)$,
  with $\Lambda_j$ homogeneous of degree $j$ in $a$.

  Taking the partial Hessian of $p$ with respect to $a$ gives,
$$
 \Haa p(a,x)[k] = 2[\Lambda_1(k,x)^T \Lambda_1(k,x)
        + L_2(k,x).
$$
  Since $\Haa p(a,x)[k]$ is negative semi-definite (in $k$
  for each $(a,x)$), and since the first term above
  is homogeneous of degree two in $k$ while the second
  has degree at most one in $k$, it follows that
  $\Lambda_1(k,x)=0$ and $L_2(k,x)$ is negative semi-definite.
  Since  $L_2(k,x)$ has degree at most one in $x$ and
  is negative semi-definite, it  does
  not depend on $x$; $L_2(a,x)=L_2(a)$.
  Further, $L_2(a)$ is negative semi-definite
  and hence can be written as $-S(a)^TS(a)$.
 \end{proof}

  We anticipate that many of the results in this section
  and the last section will {\it localize} as the
  following theorem illustrates. By $\|a\|<1$
  we mean the open matrix convex domain with
  $(\{A\in\gtupn : \sum A_j^2 < I_n\})_n \subset \gtup$.

\begin{theorem}
 \label{thm:convex/concave-local}
   Suppose $p(a,x)$ is symmetric. If
  \begin{itemize}
    \item[(i)] $p(a,x)$ is convex in $x$ for $\|a\|<1$;
    \item[(ii)] $p(a,x)\succeq 0$ for $\|A\|<1$ and all $X$; and
    \item[(iii)] $p$ is concave in $a$,
  \end{itemize}
   then $p$ has the form,
$$
  p(a,x)=  W(x)^T (R(a)-Q(a)) W(x),
$$
   where
  \begin{itemize}
   \item[(a)] $W(x)$ is the vector with $g+1$
      entries the  monomials
         $\{\emptyset, x_1,\dots,x_g\}$;
   \item[(b)]  $R(a)$ and $Q(a)$ are symmetric
        matrix polynomials in $a$;
   \item[(c)] $R(a)$ has degree at most one in $a$;
   \item[(d)] $Q(a)$ is homogeneous of degree two and $Q(A)\succeq 0$
 for
    every $A$ (and so $Q(a)$ is a sum of squares); and
   \item[(e)] $R(A)-Q(A)\succeq 0$ for $\|A\|<1$.
  \end{itemize}

    Note that the converse is true too; i.e., if $p$ has
    the form above and (a)-(e) are satisfied, then $p$
    satisfies (i)-(iii).
\end{theorem}

\begin{proof}
  The convexity in $x$ and concavity in $a$ are enough
  to imply that $p$ has degree at most two in each
  of $x$ and $a$ (so at most degree four). Further,
  by convexity in $X$ for $\|A\|<1$, it follows
  from Corollary \ref{cor:convexinxforainU-again} that
$$
  p(a,x)= \frac12 V(a)[x]^T Z(a) V(a)[x] + L(a,x),
$$
  where $Z(A)\succeq 0$ for $\|A\|<1$ and
  $V(x)[h]$ is the border vector (with respect to $x$) for $p(x,a)$,
 which is
  linear in $x$ (homogeneous of degree one) and
  $L(a,x)$ has degree at most one in $x$.  In particular,
  by considering the differing degrees of the $x$ terms,
  there can be no cancellation between the two terms.
  Because of this lack of cancellation, both terms
  have degree at most two in $a$.

  Represent $V(x)[a]= V_0(x) \oplus V_1(x)[a]$ and decompose
  $Z$ with respect to this direct sum as
$$
 Z(a)=\begin{pmatrix} M_{00}(a) & M_{01}(a)\\ M_{01}^T(a) & M_{11}
      \end{pmatrix}.
$$
  Note that $M_{11}$ is constant, $M_{01}(a)$ has degree at most one,
  and $M_{00}(a)$ has degree at most two.   Since $Z(A)\succeq 0$
  for $\|A\|<1$, it follows that $M_{11}\succeq 0$.

  We now take the Hessian of $p$ with respect to $a$,
$$
 \frac12 \Haa p(a,x)[k] = V(x)[k]^T \begin{pmatrix} 2
 M_{00}^{\prime\prime}[k]
       & 2 M_{01}^\prime[k] \\ M_{10}^\prime [k] & 2 M_{11}
 \end{pmatrix} V(x)[k]
            + \Haa L(x)[k].
$$
  Replacing $x$ with $tx$, choosing $t$ large, and using the
  hypothesis that $p(a,x)$ is concave in a, so that $\Haa p(A,X)[K]
 \preceq 0$ for all choices
  of $X,A,K$ it follows that the first term above is negative
  semidefinite; i.e.,
$$
V(X)[K]^T \begin{pmatrix} M_{00}^{\prime\prime}[K]
       & M_{01}^\prime[K] \\ M_{10}^\prime [K] & M_{11}
 \end{pmatrix}V(X)[K]
     \preceq 0
$$
  for all $X,A,K$. Thus, by Lemma \ref{lem:CHSY-in-action} (really
  it may first be necessary to take direct sums), the matrix
 $$\begin{pmatrix} M_{00}^{\prime\prime}[K]
       & M_{01}^\prime[K] \\ M_{10}^\prime [K] & M_{11} \end{pmatrix}
$$
  is negative definite and,
  therefore, $M_{11} \preceq 0$.  However, as noted above, $M_{11}\succeq 0$,
  and thus $M_{11}=0.$

  Since $Z(A)\succeq 0$ for $\|A\|<1$
  and $M_{11}=0$, it follows that $M_{01}=0$, too. We conclude that
  $p$ has the form,
 \begin{equation}
  \label{eqn:vzv1}
  p(a,x)= V(x)^T Z(a) V(x) + L(a,x),
 \end{equation}
   where $V(x)$ is the vector with entries $\{x_1,\dots,x_g\}$;
   $Z(a)$ (is a matrix and has) degree at most two; $Z(A)\succeq 0$
   for $\|A\|<1$; and $L(a,x)$ has degree at most two in $a$ and
   at most one in $x$.

  Coming at things from the other way, the fact that $p(a,x)$ is
 concave
  (globally)in $a$ (for each $x$) implies that
 \begin{equation}
  \label{eqn:vzv2}
   p(a,x)= -\Lambda(a,x)^T \Lambda(a,x) + M(a,x),
 \end{equation}
   where $\Lambda(a,x)$ is linear in $a$ and $M(a,x)$ has degree at
 most one in $a$.
   As usual, the fact that $p$ also has degree at most two in $x$
 implies
   that $\Lambda(a,x)$ has degree at most one in $x$ and $M(a,x)$ has
   degree at most two in $x$.

   Comparing the representations of equations \eqref{eqn:vzv1} and
 \eqref{eqn:vzv2}
   with an eye toward the terms which are homogeneous of degree two
   in each of $a$ and $x$, it follows that $\Lambda(a,x)$ is a linear
   combination of terms of the form $a_j$ and $a_jx_\ell$.
 Consequently,
   there is a matrix-valued $Q(a)$ so that
 \begin{equation*}
   \Lambda(a,x)^T \Lambda(a,x) = W(x)^T Q(a) W(x),
 \end{equation*}
    where $Q(a)$ is a sum of squares, and $W(x)$ is the
   vector (of polynomials) defined in the statement of the Theorem with
   entries $\{\emptyset, x_1,\dots,x_g\}$.

   From equation \eqref{eqn:vzv1}, $p(a,x)$ does not contain monomials
  of
   the form $x_jx_ka_\ell$ (or $a_\ell x_k x_j$). Hence, $M(a,x)$ can
 be written
   as
\begin{equation*}
   M(a,x)= W(x)^T R(a)W(x),
\end{equation*}
   where $R(a)$ has degree at most one.  Thus,
\begin{equation*}
   p(a,x)=W(x)^T(R(a)-Q(a))W(x),
\end{equation*}
   where $R,Q$ satisfy the conditions (a) - (d).

   To complete the proof we need to show $R(a)-Q(a)\succeq 0$ for
 $\|a\|<1$.

  For a given $n$, the set
 \begin{equation*}
  \Gamma_n  = \{W(X)v=\begin{pmatrix} v\\ X_1 v\\ \vdots \\ X_g
 v\end{pmatrix}
          : v\in\mathbb R^n, \ \ X\in\gtupn, \ \ v\in\mathbb R^n\}
 \end{equation*}
   is dense in $\mathbb R^{(g+1)n}$.   To prove this claim we follow
  the route of the CHSY-Lemma (see Appendix \ref{appendix:CHSY}).
 Suppose first that
   $v=e_1$ and let
 \begin{equation*}
   w=\begin{pmatrix} w_1 \\ \vdots \\ w_g \end{pmatrix} \in \mathbb
 R^{gn}
 \end{equation*}
  be given. Letting
 \begin{equation*}
   X_j =\begin{pmatrix}  w_{j1} & w_{j2} & \dots \\
                         w_{j2} &  0 &\dots \\
                         \vdots &  0 & 0 \end{pmatrix},
 \end{equation*}
   so that $X_j=X_j^T$ and $X_j e_1= w_j$, it follows that
$$
\Gamma_n \supset W(X)e_1 =\begin{pmatrix} e_1 \\ w \end{pmatrix}.
$$

  It follows that $\Gamma_n$ contains all vectors of the form,
$$
  \begin{pmatrix} w_0\\w_1 \\ \vdots \\ w_g \end{pmatrix} \in \mathbb
 R^{(g+1)n}
$$
  for which $w_0\ne 0$ and the claim follows.

  Finally, since $p(A,X)\succeq 0$ for $\|A\|<1$ and all $X$, if
 $\|A\|<1$
  and $X\in\gtupn$ and $v\in\mathbb R^n$, then
$$
  0\le p(A,X)v,v\rangle =  \langle (Q(A)-P(A))W(X)v, W(X)v\rangle.
$$
  From the density of $\Gamma_n$ it follows that $Q(A)-P(A)\succeq 0$
\end{proof}

\section{The Polynomial Congruence}
 \label{sec:poly-congruence}
  In this section we establish a polynomial congruence between
  $Z(a,x)$ and $\mathcal Z(a)$ (the derived
  matrix, $\mathcal Z(a)=Z(a,0)$) in the case that
  $Z(a,x)$ is the middle matrix of the partial
  Hessian of  a symmetric polynomial $p(a,x)$.
  The relation is analogous to that found in
  \cite{DHM}

  Suppose $p(a,x)$ is a polynomial of degree $d_x$ in $x$ and $d_a$ in
  $a$.  The partial Hessian of $p$ can be written as a sum of terms of
  the form $m_Lh_Lrh_Rm_R$, where $r = {Z}(m_Lh_L,h_Rm_R)$, and $m_L$
  and $m_R$ are monomials in $x$ and $a$.  We use $(m_Lh_L, h_Rm_R)$
  to index these terms.
  We will further write $m_L$ (and
  similarly $m_R$) in the form
$$
 x_{\ell_{1}}m_{\ell_{1}} \cdots
  x_{\ell_{i}}m_{\ell_{i}},
$$
  where each $m_{\ell_s}$ is a monomial in $a$
  alone. So in particular, here $m_L$ has degree $i$ in $x$.

  Let ${Z} = [{Z}_{ij}]$ be the middle matrix
  for the partial Hessian of $p$.  Here the block structure
  indicated by $Z_{ij}$ is determined by the the degree in
  $x$ of the monomials in the middle matrix per the usual
  convention.

  For each $j = 0, \dots ,d_x - 2$, let $K_j$ be the $N_j \times
  N_{j-1}$ matrix with entries
$$
  K_j(h_{k_{j+2}}m_{k_{j+2}}x_{k_{j+1}}m_{k_{j+1}} \cdots
  x_{k_{1}}m_{k_{1}}, h_{k_{j+1}}m_{k_{j+1}}x_{k_{j}} \cdots
  x_{k_{1}}m_{k_{1}}) = x_{k_{j+2}}m_{k_{j+2}},
$$ and all others being
  $0$.

\begin{lemma}
 \label{lem:recursion}
  Let $p$ be a polynomial of degree
  $d_x \ge 2$ in $x$ and of arbitrary degree in $a$. Let ${Z} =
  [{Z}_{ij}]$ be the middle matrix for the Hessian of $p$. Then
  ${Z}_{i,j+1}K_j + {Z}_{i,j}(a,0) = {Z}_{i,j}$ for $i = 0, \dots,
  d_x-2$, and $j = 0, \dots ,d_x-3$, where $i+j \le d_x-2$.
\end{lemma}

\begin{proof}
  It suffices to prove the result for monomials, since
  it is evidently linear.  Thus, suppose $i+j<d_x-2$
  and the monomial
 $$x_{\ell_{1}}m_{\ell_{1}} \cdots
  x_{\ell_{i}}m_{\ell_{i}}h_{\ell_{i+1}}r
  h_{k_{j+2}}m_{k_{j+2}}x_{k_{j+1}}m_{k_{j+1}} \cdots
  x_{k_{1}}m_{k_{1}}
 $$
  is the $(x_{\ell_{1}}m_{\ell_{1}} \cdots
  x_{\ell_{i}}m_{\ell_{i}}h_{\ell_{i+1}},
  h_{k_{j+2}}m_{k_{j+2}}x_{k_{j+1}}m_{k_{j+1}} \cdots
  x_{k_{1}}m_{k_{1}}$)-term of the Hessian
  (so this is an entry of $Z_{i,j+1}$).  Then the
  $$(x_{\ell_{1}}m_{\ell_{1}} \cdots
  x_{\ell_{i}}m_{\ell_{i}}h_{\ell_{i+1}},
  h_{k_{j+1}}m_{k_{j+1}}x_{k_{j}} \cdots x_{k_{1}}m_{k_{1}})$$ term
  will be
 $$
  x_{\ell_{1}}m_{\ell_{1}} \cdots
  x_{\ell_{i}}m_{\ell_{i}}h_{\ell_{i+1}}
  rx_{k_{j+2}}m_{k_{j+2}} h_{k_{j+1}}m_{k_{j+1}}x_{k_{j}} \cdots
  x_{k_{1}}m_{k_{1}}.
 $$
  In other words,
  ${Z}_{i,j}(x_{\ell_{1}}m_{\ell_{1}} \cdots
  x_{\ell_{i}}m_{\ell_{i}}h_{\ell_{i+1}},
  h_{k_{j+1}}m_{k_{j+1}}x_{k_{j}} \cdots x_{k_{1}}m_{k_{1}})$ equals
 $$
  {Z}_{i,j+1}(x_{\ell_{1}}m_{\ell_{1}} \cdots
  x_{\ell_{i}}m_{\ell_{i}}h_{\ell_{i+1}},
  h_{k_{j+2}}m_{k_{j+2}}x_{k_{j+1}}m_{k_{j+1}} \cdots
  x_{k_{1}}m_{k_{1}})x_{k_{j+2}}m_{k_{j+2}}.
$$

  Now, the $(x_{\ell_{1}}m_{\ell_{1}} \cdots
  x_{\ell_{i}}m_{\ell_{i}}h_{\ell_{i+1}},
  h_{k_{j+1}}m_{k_{j+1}}x_{k_{j}} \cdots x_{k_{1}}m_{k_{1}})$-entry of
  ${Z}_{i,j+1}K_j$ is the product of row
 $$
  x_{\ell_{1}}m_{\ell_{1}}\cdots
  x_{\ell_{i}}m_{\ell_{i}}h_{\ell_{i+1}}
 $$
  of ${Z}_{i,j+1}$ and column
 $$
  h_{k_{j+1}}m_{k_{j+1}}x_{k_{j}} \cdots x_{k_{1}}m_{k_{1}}
 $$
  of $K_j$.  The only nonzero entry of column
  $h_{k_{j+1}}m_{k_{j+1}}x_{k_{j}}m_{k_{j}} \cdots x_{k_{1}}m_{k_{1}}$
  of $K_j$ is $K_j(h_{k_{j+2}}m_{k_{j+2}}x_{k_{j+1}}m_{k_{j+1}} \cdots
  x_{k_{1}}m_{k_{1}}, h_{k_{j+1}}m_{k_{j+1}}x_{k_{j}} \cdots
  x_{k_{1}}m_{k_{1}}) = x_{k_{j+2}}m_{k_{j+2}}.$ Hence, the
  $(x_{\ell_{1}}m_{\ell_{1}} \cdots
  x_{\ell_{i}}m_{\ell_{i}}h_{\ell_{i+1}},
  h_{k_{j+1}}m_{k_{j+1}}x_{k_{j}} \cdots x_{k_{1}}m_{k_{1}})$-entry of
  ${Z}_{i,j+1}K_j$ is
 $$
  {Z}_{i,j+1}(x_{\ell_{1}}m_{\ell_{1}} \cdots
  x_{\ell_{i}}m_{\ell_{i}}h_{\ell_{i+1}}m_{\ell_{i+1}},
  h_{k_{j+2}}m_{k_{j+2}}x_{k_{j+1}}m_{k_{j+1}} \cdots
  x_{k_{1}}m_{k_{1}})x_{k_{j+2}}m_{k_{j+2}},
 $$
  which equals
 $$
  {Z}_{i,j}(x_{\ell_{1}}m_{\ell_{1}} \cdots
  x_{\ell_{i}}m_{\ell_{i}}h_{\ell_{i+1}},
  h_{k_{j+1}}m_{k_{j+1}}x_{k_{j}} \cdots x_{k_{1}}m_{k_{1}}).
 $$
  We conclude that ${Z}_{i,j+1}K_j + {Z}_{i,j}(0) = {Z}_{i,j}$
  whenever
  $i + j < d_x-2$.

   If $i + j = d_x-2$, then ${Z}_{i,j+1} = 0$ and
  ${Z}_{i,j} = {Z}_{i,j}(a,0)$, so that ${Z}_{i,j+1}K_j + {Z}_{i,j}(a,0) =
  {Z}_{i,j}$. Clearly the result  also holds when $i + j > d_x-2$.
\end{proof}

To illustrate, in Example \ref{ex:x^2axplus}, note that
${Z}_{00}(a,x) = {Z}_{01}(a,x) K_1 + {Z}_{00}(0,a)$, where

$$K_1 = \left[ \begin{array}{cc}
x & 0 \\
0 & 0 \\
xa & 0 \end{array}\right].$$

\begin{theorem}
 \label{thm:congruence}
  There is a matrix polynomial $A(a,x)$ so that
$$
  Z(a,x)A(a,x)=Z(a,0).
$$
  Further, $A$ has a square root $B$, so that $B^2=A$,
  which is also a polynomial for which
$$
  B^T(a,x)Z(a,x)B(a,x)=Z(a,0).
$$
  Further, $B$ is invertible (and its inverse is a polynomial).
  In particular, for any $(A,X)\in\ggtup$, $B(A,X)$ is invertible.
\end{theorem}

With Lemma \ref{lem:recursion} in place,
the proof of Theorem \ref{thm:congruence} follows along the lines of
 the proof of Theorem
7.3 in \cite{DHM}.  Let $A(a,x)$ denote the matrix
\begin{equation}
A(a,x):= \left[\begin{array}{ccccc}
I&0&\cdots&0&0\\
-K_0&I&\cdots&0&0\\
0&-K_1&\cdots&0&0\\
\vdots&\vdots& &\vdots&\vdots\\
0&0&\cdots&I&0\\
0&0&\cdots&-K_{d_x-3}&I\end{array} \right].
\end{equation}
and observe  $Z(a,x)A(a,x)=Z(a,0)$ follows from Lemma
\ref{lem:recursion}.  For $B(a,x)$, one takes

$$
   B(a,x):= \sum_{j=0}^{d_x-2} \begin{pmatrix} \frac12 \\
 j\end{pmatrix}
   (A(a,x)-I)^j.
$$

\section{Signatures}
 \label{sec:sigs}
 Since the partial Hessian, $\Hxx p(a,x)[h],$ of the symmetric
 polynomial $p$ is symmetric, it has a sum of difference of squares
 (SDS) decomposition,
$$
  \Hxx p(a,x)[h] = \sum_{j=1}^n q_j^T(a,x)[h] q_j(a,x)[h]
   -\sum_{\ell=1}^m r_\ell^T(a,x)[h] r_\ell(a,x)[h],
$$
 where  $q_j, r_\ell$ are linear in $h$.

 The minimum number of positive (resp. negative) squares
 needed in such a decomposition
 is the \emph{positive (resp. negative) signature} of $\Hxx p(a,x)[h]$,
 denoted
 $\sigma_\pm (\Hxx p(a,x)[h])$.  More generally, any symmetric
 polynomial $q(a,x)[h]$ which is homogeneous of degree two in $h$
 has a SDS decomposition (again with factors $q_j, r_\ell$ linear
 in $h$).  Accordingly, we may define the signature $\sigma_\pm(q)$.

 For a symmetric matrix polynomial $Z(y)$ in $y$ and $Y\in\gtupn$
 we let $\mu_\pm (Z(Y))$ denote the number of positive or negative
 eigenvalues of the symmetric matrix  $Z(Y)$.

 The following proposition is a generalization of a result from \cite{DHM}.

\begin{proposition}\label{sig}
 Let $q(a,x)[h]$ be a symmetric polynomial in nc variables
 $(a,x,h)=(a_1, \ldots , a_{g}, x_1,\ldots,x_{g},h_1,\ldots,h_{g})$
 that is of degree $\ell$ in
 $x$ and homogeneous of degree two in $h$ with middle matrix
 $Z(a,x)$. Then
\begin{equation*}
  \mu_\pm(Z(A,X)) \leq n {\sigma_\pm}(q)
\end{equation*}
 for each $n$ and $(A,X)$ in $\ggtupn.$
\end{proposition}

\begin{proof}
 Suppose
\begin{equation*}
 q(a,x)[h] = \sum_{j=1}^{\sigma_+}
  f_j^T f_j - \sum_{j=1}^{\sigma_-} g_j^T g_j,
\end{equation*}
 where the $f_j$ and $g_\ell$ are linear in $h$,
 is a SDS
 decomposition with the minimum number of positive squares. Then
$$
 f_j^T(a,x)[h] f_j(a,x)[h] =  V(a,x)[h]^T F_j(a,x) V(a,x)[h]
$$
and
$$
  g_j^T(a,x)[h] g_j(a,x)[h] =  V(a,x)[h]^T G_j(a,x) V(a,x)[h]\,,
$$
 where each of the matrix polynomials $F_j(a,x)$ and $G_j(a,x)$ on
 the right is of the form
\begin{equation*}
 F_j(a,x)=\Phi_j(a,x)\Phi_j(a,x)^T \quad \textrm{and}\quad G_j(a,x) =
  \Psi_j(a,x) \Psi(a,x)^T\,,
\end{equation*}
 for vectors  $\Phi_j$ and $\Psi_j$ whose entries are polynomials of
 degree less than or equal to $\ell$ and the border vector
 $V(a,x)[h]$ is linear in $h$.

  Since
$$
 q(a,x)[h] = V(a,x)[h]^T ( \; \sum_{j=1}^{\sigma_+} F_j(a,x)
   -\sum_{j=1}^{\sigma_-} G_j(a,x) \; ) V(a,x)[h],
$$
  and for a given $q$ the middle matrix is unique, once the border
  vector is fixed, it follows that
$$
 Z(a,x) = \sum_{j=1}^{\sigma_+} F_j(a,x) - \sum_{j=1}^{\sigma_-}
   G_j(a,x).
$$
  Since the rank of $F_j(A,X)$ is at most $n$, it follows that $Z(A,X)$
  has at most $n\sigma_+^{min}$ positive eigenvalues.
\end{proof}

\begin{corollary}\label{cor:sig}
  Suppose $q$ is the partial Hessian of a symmetric
  polynomial of degree $d_x \ge 3$ in $x$.
  If $\sigma_\pm(q) \le 1$, then
  \begin{itemize}
  \item[(i)] $\sigma_\pm =1$;
  \item[(ii)]
 $$
   \sup_n\left\{\frac{\mu_\pm(Z(A,0))}{n}: A \in \gatupn
   \right\}=\sigma_\pm(q).
  $$
\end{itemize}
\end{corollary}

\begin{proof}
   If $\sigma_+(q)=0$, then $q$ is the negative of a sum of squares, so
   that earlier results then imply that
   the degree of $p$ is two.    Thus,
   $\sigma_+(q)=1$.
   In this case,
   in view of the previous proposition, it suffices to prove
   that there is an $n$ and an $A\in\gtupn$ such that
   $\mu_+(Z(A,0))\ge n$.
   The $Z_{0,d_x-2}(a,x)=Z_{0,d_x-2}(a,0)$
   block of the middle
   matrix of the Hessian of $p$
   is not zero. It therefore
   has an entry (NC polynomial)
   which is not zero  and therefore, by the Guralnick-Small
   Lemma (Lemma \ref{lem:G-S}), there is an $n$
   and an $A \in \mathbb{S}_n(\mathbb R ^g)$ so
   that $Z_{0,d-2}(A,0)$ has rank at least $n$. Since the middle
   matrix is symmetric and {\it upper anti-diagonal} and the upper right
   has rank $n$, it has at least $n$ positive eigenvalues (Corollary 5.4 in
   \cite{DHM}).  This gives item (ii).
\end{proof}

\begin{conjecture} Suppose $q$ is the partial Hessian of a symmetric
 polynomial.
 Then $$\sup_n\left\{\frac{\mu_\pm(Z(0,A))}{n}: A \in \gatupn
 \right\}=\sigma_\pm(q).$$
\end{conjecture}

\section{Appendix A. The [CHSY] Lemma}
 \label{appendix:CHSY}
 At the root of the \cite{CHSY} Lemma is the following
\begin{lemma}
 \label{lem:CHSY-start}
  Fix $n>d.$ If $\{z_1,\dots,z_d\}$ is a linearly independent
  set in $\mathbb R^n$, then the codimension of  \begin{equation*}
   \left\{ \begin{pmatrix} Hz_1 \\ Hz_2 \\ \vdots \\ Hz_d
 \end{pmatrix}:
      H \in \mathcal{S}_n(\mathbb{R}) \right\} \subset \mathbb R^{nd}
  \end{equation*}
   is $\frac{d(d-1)}{2}$. In particular, this codimension is
 independent
   of $n$.
\end{lemma}

\begin{proof}
   Consider the mapping $\Phi$ given by
 \begin{equation*}
  \mathcal{S}_n(\mathbb{R})  \ni H \mapsto
      \begin{pmatrix} Hz_1 \\ Hz_2 \\ \vdots \\ Hz_d \end{pmatrix}.
 \end{equation*}

   Since the span of $\{z_1,\dots,z_d\}$ has dimension $d$, it follows
 that
   the kernel of $\Phi$ has dimension $\kappa=\frac{(n-d)(n-d+1)}{2}$
 and
   hence the range has dimension $\frac{n(n+1)}{2}-\kappa$.  To see
 this
   assertion, it suffices to assume that the span of
 $\{z_1,\dots,z_d\}$
   is the span of $\{e_1,\dots,e_d\}\subset \mathbb R^n$ (the first
  $d$ standard basis vectors in $\mathbb R^n$) in which case $H$ is symmetric
 and
   $Hz_j=0$ for all $j$  if and only if
 \begin{equation*}
    H=\begin{pmatrix} 0 & 0 \\ 0 & H^\prime \end{pmatrix},
 \end{equation*}
    where $H^\prime$ is symmetric and $(n-d)\times (n-d)$.

   Finally, we conclude
   that the codimension of the range is
 \begin{equation*}
    nd -(\frac{n(n+1)}{2}-\kappa) = \frac{d(d-1)}{2}.
 \end{equation*}
\end{proof}

\begin{lemma}\cite{CHSY}
 \label{lem:CHSY}
 If $n>d$ and $\{z_1,\dots,z_d\}$ is a linearly independent
 subset of $\mathbb R^n$, then the codimension of
\begin{equation*}
  \{ \oplus_{j=1}^g \begin{pmatrix} H_jz_1 \\ H_jz_2 \\ \vdots \\
 H_jz_d \end{pmatrix}:
     H=(H_1,\dots,H_g)\in \gtupn\} \subset \mathbb R^{gnd}
\end{equation*}
 is $g\frac{d(d-1)}{2}$ (independent of $n$).
\end{lemma}

  Finally, the form in which we generally apply the lemma is the
 following.
\begin{lemma}
 \label{lem:CHSY-in-action}
  Fix $n,$ $v\in\mathbb R^n$,  and $A\in \gatupn$ and $X\in \gxtupn$.
  If the set $\{m(a,x)v: |m| \le d\}$ (bi-degree of $m$ is at most
  $d$, meaning degree at most $d_a$ in $a$ and $d_x$ in $x$) is
 linearly
  independent, then the codimension of
$$
  \{V(A,X)[H]v: H\in \gxtupn\}
$$
 is at most $g\frac{\kappa(\kappa-1)}{2}$, where
  $\kappa = \sum^{d_a+d_x} g^j$ and where $V$ is the
  border vector associated to the given collection of monomials.
  Again, this codimension is independent of $n$.
\end{lemma}

\begin{proof}
  Let $z_m = m(A,X)v$ for the given collection of monomials $m$.
  There are at most $\kappa$ of these. Now apply the previous lemma.
\end{proof}

\section{Appendix: Generic Invertibility}
\label{sec:GS}
In this Appendix we give the proof of Lemma \ref{lem:G-S} supplied
to us by  R. Guralnick and L. Small.

\proof
(1) \ By a result of Amitsur \cite{A69}[Theorem 1] , Theorem 1) ,  if $q$ vanishes on all
symmetric matrices of size m and $deg \ q = d$, then all matrices of
size $m$ satisfy the standard identity $S_{2d}$, but this bounds the
size of the matrices. So  taking $n$ sufficiently large, $p$ does not
vanish on all $n \times n$ symmetric matrices.

(2) \ Now consider   $n \times n$ matrices $Y_1, ..., Y_g$ with
distinct commuting variables as entries (called ``generic matrices"
by some) e.g.., $Y_1= ( y_{ij} )_{i,j= 1, \cdots, n} $ and their
transposes $Y_1^T, ..., Y_g^T$. Consider the algebra these matrices
of polynomials  generate. If $n$ is a power of 2, then this is an
integral  domain with a quotient division ring $D$ contained in the
$n \times n$ matrices over
the field of rational functions.
See for example, \cite{BS88}.

(3) \ To finish the argument:  \\
take your $q$ and consider $q(Y_1+Y_1^T, ..., Y_g +Y_g^T)$ for $n$ a
sufficiently large power of $2$ (depending only the degree of $q$), so
$q$ is nonzero (note that since we are in characteristic not $2$, the
$Y_i + Y_i^T $ are generic symmetric matrices, so $q$ not zero means
that $q$ does not vanish with those arguments) and $q$ is an element
of the ring generated by $Y_1,...Y_g, Y_1^T,...Y_g^T$ and this is a
domain with a quotient division ring, it follows that
$\det q(Y_1+Y_1^T, ..., Y_g +Y_g^T)$ is a polynomial which is not
identically zero. Thus it is nonzero on an open dense set, so the
result follows. \qed

\end{document}